\documentclass[twoside,11pt,reqno]{amsart}
\usepackage{amsmath,amssymb,amscd,mathrsfs, verbatim}

\newtheorem{propo}{Proposition}[section]

\newtheorem{lemma}[propo]{Lemma}
\newtheorem{corol}[propo]{Corollary}
\newtheorem{theor}[propo]{Theorem}

\newcommand{\lan}{ \langle }
\newcommand{\ran}{ \rangle }

\newcommand{\diag}{\mathop{\rm diag}\nolimits}

\newcommand{\Id}{\mathop{\rm Id}\nolimits}

\newcommand{\al}{\alpha}

\def\d12{{_{12}}}

\def\acf{{algebraically closed field }}

\def\ii{{if and only if }}

\def\po{{polynomial }}

\def\F{{\rm F}}

\newcommand{\el}{\end{lemma}}

\begin{document}

\title{On generators and representations of the sporadic simple groups}
\author{L. Di Martino, M.A. Pellegrini and A.E. Zalesski}

\thanks{The second author was supported by FEMAT}

\begin{abstract}
In this paper we determine the irreducible projective representations of sporadic simple groups over an arbitrary algebraically closed field $F$, whose image contains an almost cyclic matrix of prime-power order. A matrix $M$
is called cyclic if its characteristic and minimum polynomials coincide, and we call $M$ almost cyclic if, for a suitable $\alpha\in F$, $M$ is similar to
 $\diag(\al\cdot \Id_h, M_1)$, where $M_1$ is cyclic and $0\leq h\leq n$.
The paper also contains results on the generation of sporadic simple groups by minimal sets of conjugate elements.  
\end{abstract}


\maketitle
\section{Introduction}

Problems on group generation by suitable subsets have been the subject of intensive research since the origins of group theory. Apart from its intrinsic interest, this subject gives rise to many applications, and has been used extensively in answering questions on many topics within group theory. In particular, it is well known that  some aspects of representations of finite groups are connected to the existence of  generating sets of a certain kind. 

In this paper we are interested in the generation of a group by sets of conjugates of a given group element. Concerning the sporadic simple groups, one of the first results related to generation by conjugates is due to I. Zisser (\cite{Zis}), who determined the 'covering number' of each sporadic group (see also Table 1 on p. 554 in  \cite {GS}). This is the minimum number $m=m(G)$ such that for every non-trivial conjugacy class $C$ of $G$ one has $G=\bigcup_{i=1}^m C^i$, where $C^i= \{g \in G:g= x_1x_2\cdots x_i$ for some $ x_j \in C \}$. If $m$ is the covering number of a simple group $G$, then $G$ can be generated by $m+1$ suitable elements from any given non-trivial class $C$ (see  \cite {GS}, Lemma $2.12$). It was shown in \cite{Zis} that $m(G) \leq 4$  unless $G=Fi_{22}, Fi_{23}$, in which case $m(G)=6$.

In \cite {GS} (see Lemma $7.6$ and Table 1, p. 554), these bounds were slightly improved for some of the groups, combining Zisser's results with the information provided by the Atlas of finite simple groups (\cite {Atl}, \cite {AOL}) and the knowledge of lower bounds for the degrees of representations of sporadic groups. However, no specific attention to the order of the elements of a given class $C$ was paid there, and so it remained open whether and when $G$ might be generated by, say, two elements of $C$. More recently, the problem of the generation of a simple group by sets of conjugate involutions satisfying certain specific conditions has been considered by J. Ward in his PhD dissertation (see \cite{Ward}).

In this paper we determine, for every finite sporadic simple group $G$ and most conjugacy classes $C$ of $G$, the minimum number $\alpha_G(g)$ of conjugates of  $g\in C$ required to generate $G$. The results obtained are summarized in Theorem \ref{gen} below.

\smallskip
Theorem \ref{gen} can be viewed as a refinement and improvement of the results of \cite{Zis} and \cite {GS}. Indeed, in comparison with the bounds given there, it provides better estimates for each sporadic group $G$, except when $G \in \{M_{11}, J_1, J_2, M\}$ (in the Atlas notation).

We also find out (see Section \ref{secgen}) that, unless $G=M$, whenever the order of $g \in G$ is greater than $4$, then  $\alpha_G(g) =2$, and  $\alpha_G(g)$ conjugates of $g$ can be chosen, such that their product in a suitable ordering has order equal to the highest prime divisor of $|G|$.

\smallskip
We emphasize that Theorem \ref{gen} cannot be proven via straightforward computations, except when the groups are very small. Instead, our approach makes use of character theory and the character tables of the sporadic groups, and can be outlined as follows.

Given a finite group $G$, and $k\geq 3$ (not-necessarily distinct) conjugacy classes of $G$  $C_1,\ldots, C_k$, there exists a formula in terms of the values of irreducible characters of $G$ at   $C_1,\ldots, C_k$ (see Section \ref{mach}), which gives the number $\Delta_G(C_1,\ldots,C_k)$ of solutions of the equation  $g_{1}g_{2} \cdots g_{k-1}=g_k$, where  $g_i \in C_i$ ($1 \leq i \leq{k-1}$) and $g_k$ is a fixed element of the class $C_k$. Next compute, for every maximal subgroup $H$ of $G$ that meets every class  $C_1,\ldots, C_k$, the number $\Delta_H(c_1,\ldots,c_k)$ for all the $H$-conjugacy classes $c_1,\ldots,c_k$ such that $c_i\subseteq H\cap C_i$. Denote by $\Sigma_H(C_1,\ldots,C_k)$ the sum of all such structure constants $\Delta_H(c_1,\ldots,c_k)$. Suppose that $$ \Delta_G(C_1,\ldots,C_k)  > \sum h(g_k,H)\cdot \Sigma_H(C_1,\ldots,C_k),$$ where $h(g_k,H)$ is the number of the distinct conjugates of $H$ containing $g_k$, and the sum is taken over the representatives $H$ of the $G$-classes of maximal subgroups of $G$ containing elements of all the classes $C_1,\ldots, C_k$. Then there exist elements $g_i \in C_i$ such that $G=\langle g_1,\ldots,g_{k-1}\rangle$.
In our situation, $C_1= \cdots =C_{k-1}$, and in most cases the class $C_k$ plays a special role. Namely, most often we find that  $C_k$ can be chosen uniformly, that is $C_k$ can be chosen to be the same class for any choice of $C_1$.
Recall that a group $G$ is said to be $(C_1,\ldots,C_k)$-generated if there exist $g_i \in C_i$ ($1 \leq i \leq k$) such that $g_1\cdots g_{k-1}=g_k$ and $G=\langle g_1,\ldots,g_{k-1}\rangle$. Thus, our computations yield results on $(C_1,\ldots,C_k)$-generation for $C_1= \cdots =C_{k-1}$. (See Section \ref{mach} for further details.)

\medskip

The results on generating sets, 
 described in detail in Section \ref{secgen}, allow us to study in an efficient way certain properties of the eigenvalues of matrices in the representations of the sporadic groups. More precisely, we are committed to determine all the projective irreducible representations of sporadic groups for which there exist elements of prime-power order represented by so-called almost cyclic matrices.  

The notion of almost cyclic matrix is a generalization of the notion of cyclic matrix. Namely, let $V$ be a finite dimensional vector space over a field $F$. 
Cyclic matrices are
exactly those  whose  characteristic \po coincides with the minimum
one. (Note that a matrix $X\in {\rm End} ~V$ is cyclic \ii the
$F\langle X \rangle$-module $V$ is cyclic, that is, is generated by
a single element. This is a standard terminology of ring theory, and
the source of the term `cyclic matrix'. Matrices with simple spectrum
often arising in applications are cyclic.) 

Now, we define a matrix $M\in Mat(n,F)$ to be almost cyclic if
there exists $\al\in F$ such that $M$ is similar to
 $\diag(\al\cdot \Id_h, M_1)$, where $M_1$ is cyclic and $0\leq h\leq n$.
 
Observe that, if $\overline{F}$ denotes the algebraic closure of
$F$, and  $\lambda J$ for $\lambda \in \overline{F}$  denotes a
Jordan block with eigenvalue $\lambda $,  then a matrix
$M_{1}$ is cyclic if and only if $M_{1}$ has Jordan form
$\mathrm{diag}(\lambda _{1}J_{1},....,\lambda _{s}J_{s})$, where the
$\lambda _{j}$'s, $1\leq j\leq s$, are pairwise distinct. In
particular, suppose that $M=\diag(\al\cdot \Id_h, M_1)$,  with $0\leq h<n$, is non-scalar of order $p^{a}$ for a prime $p$,  and set $\ell =\,$char$\, F$. Then $M$
is almost cyclic if and only if the eigenvalues of $M_{1}$ are pairwise
distinct when $\ell \neq p$, and if and only if $M_{1}$ consists of a single Jordan
block when $\ell =p$.

Almost cyclic matrices arise naturally in the study  of matrix groups over finite fields. 
Pseudo-reflections are important examples, as well as unipotent matrices with
Jordan  form consisting of a single non-trivial block.

A key contribution to the subject  is
a paper by Guralnick, Penttila, Praeger
and Saxl (\cite{GPPS}), in which the authors
classified linear groups over finite fields generated by `Dempwolff elements'.
Let $V=V(n,q)$ be an $n$-dimensional vector space over a finite field of order $q$,   $H=GL(V)=GL(n,q)$ and $g\in H$. We say that $g$ is a  Dempwolff element if
$|g|=p$ for some prime $p$ with $(p,q)=1$ and $g$ acts irreducibly on $V^g:=(\Id -g)V$.
U. Dempwolff in \cite{De} initiated the study of subgroups of $GL(n,q)$ generated by such elements, obtaining a number of valuable results. The main restriction in \cite{De} is the assumption that $2\dim V^g>\dim V$, and this assumption is held in  \cite{GPPS}.
Clearly, Dempwolff elements are almost cyclic (and are reflections if $p=2$).

Possibly, the strongest motivation to study groups
containing an almost cyclic matrix  is to contribute to the
recognition of linear groups and finite group representations
by a property of a single matrix. Answers to problems of this kind are often required in several applications.

In fact, there is an extensive literature containing important results related more or less strictly to our subject, both before and after Dempwolff's work (e.g. results due to Hering, Wagner, Suprunenko, Huffman, Wales, Tiep, Guralnick, Saxl and others). For a more detailed description of this literature, see, e.g., \cite {TZ00}, \cite{Z09}  and  \cite{DMZ1}.

The present paper may be viewed as a necessary piece of a project initiated in \cite{DMZ1} and
\cite{DMZ2}. 
 The paper \cite{DMZ1} classifies the irreducible cross-characteristic representations of finite quasi-simple groups of Lie type, for which  there exist unipotent elements represented by almost cyclic matrices.  The paper \cite{DMZ2} analyzes the occurrence of almost cyclic semisimple elements of prime-power order in cross-characteristic representations of finite quasi-simple groups of Lie type. 

Our goal here is to examine the irreducible representations of the finite simple sporadic groups and their covering groups. The techniques exploited are of computational nature, and thus differ substantially from those of \cite{DMZ1} and
\cite{DMZ2}. As it should be expected, substantial use is made of the mass of information available in the Atlas and Modular Atlas of finite groups   (\cite{Atl}, \cite {AOL}), together with the routines existing or implementable in GAP and MAGMA (\cite{GAP}, \cite{MGM}). The results we have obtained are collected in Section \ref{repres}.

Finally, we note that the connection between the two problems we address in the paper (generation by conjugates, existence of elements representable by almost cyclic matrices) is based on Lemma \ref{2.1} below, which bounds from above the degree of a linear group $G$ generated by almost cyclic matrices conjugate to a given $g \in G$, in terms of the order of $g$ and $\alpha_G(g)$. In fact, this was our original motivation for studying the above generation problem in detail. Furthermore, Lemma \ref{2.1}  together with other machinery described in Section \ref{mach} are essential in order to reduce significantly the amount of computations necessary to obtain  the results stated in Section \ref{repres}.

\medskip

{\bf Notation}. 
Throughout the paper we assume $F$ to be an \acf of characteristic  $\ell$.


For an $(n\times n)$-matrix $A$ over a field $F$, we denote by $m_A(x)$ and $p_A(x)$ the minimum and the characteristic polynomial of $A$, respectively.

Following the conventions introduced in the Atlas  (\cite{Atl}), we denote by $nX$ a conjugacy class of a group $G$ consisting of elements of order $n$. We warn that the letter $X$ is chosen according to the labelling adopted in GAP (\cite{GAP}). The notation used for the 26 simple sporadic groups is the standard one. The known maximal subgroups of each sporadic group can be found in \cite{W}. In the text and tables of the present paper, a representation of a given sporadic group will usually be indicated only by its degree. We emphazise here that this shortcut is justified by the fact that,  when the group has two or more representations of the same degree, the results we will obtain turn out to be independent of the choice of the representation.

\section{Basic machinery}\label{mach} 
The following elementary result establishes a useful connection between the occurrence of almost cyclic matrices in representations of irreducible linear groups and their generation by conjugates. 
\begin{lemma}\label{2.1}
Let $F$ be an algebraically closed field. If $G< GL(n,F)$ is a finite irreducible linear group generated by $m$ almost cyclic elements $g_i$ of the same order $d$ (modulo $Z(G)$), then $$n\leq m(d-1).$$ 
\end{lemma}

\begin{proof}
Let $V=V(n,F)$ be the underlying vector space of $GL(n,F)$. Let $\alpha_i$ be an eigenvalue of $g_i$ with eigenspace of maximal dimension. Define $V_i=Im(g_i-\alpha_iI)$. Clearly, $V_i$ is $\langle g_i\rangle$-invariant. Moreover, considering the action of $g_i$ induced on the quotient space $V/V_i$, we observe that  $g_i(V_i+x)= V_i+\alpha_i x$. 
Thus, for each $i=1,\ldots, m$, 
$$g_i(\sum_{j=1}^m V_j)=g_i(V_i+\sum_{j\neq i} V_j)\subseteq V_i + \alpha_i \sum_{j\neq i} V_j=\sum_{j=1}^m V_j.$$
This means that $\sum_{j=1}^m V_j$ is $G$-invariant, whence, as $G$ is irreducible, $V=\sum_{j=1}^m V_j$.
On the other hand, we  claim that, for each $i$,  $\dim V_i\leq d-1$, whence $n\leq m(d-1)$, as required. Let ${g_i}^d=\lambda I_n$, for some $\lambda \in F$.  First, suppose that $g_i$ is cyclic. Then $m_{g_i}(x)=p_{g_i}(x)$ divides $x^d-\lambda$, which implies $n\leq d$. It follows that $\dim V_i=n-\dim Ker(g_i-\alpha_iI)=n-1\leq d-1$. Next, suppose that $g_i$ is almost cyclic, but not cyclic. This means that $g_i$ is similar to a matrix of shape $\diag(\alpha_i, \ldots,\alpha_i, \bar g)$, where $\bar g$ is a cyclic matrix of size $k>0$. We have two possibilities. First, $\alpha_i$ is not an eigenvalue of $\bar g$. Then $m_{g_i}(x)=(x-\alpha_i)m_{\bar g}(x)$ divides $x^d-\lambda$, whence $1+k\leq d$ and $\dim V_i=n-(n-k)=k\leq d-1$. Next, suppose that $\alpha_i$ is an eigenvalue of $\bar g$. Then $\bar g$ is similar to a matrix of shape $\diag(J, \tilde g)$, where $J$ is the (unique) Jordan block corresponding to the eigenvalue $\alpha_i$ and $ \tilde g$ is cyclic. Let $t$ be the size of $J$. Then $m_{g_i}(x)=(x-\alpha_i)^t m_{\tilde g}(x)$ divides $x^d-\lambda$. This implies $k\leq d$ and therefore again $\dim V_i=n-(n-k+1)=k-1\leq d-1$. So the statement is proven. \end{proof}

An immediate consequence of the previous Lemma is that, if $\Phi: G\rightarrow GL(n,F)$ is an irreducible faithful representation of a finite group $G$, which can be generated by $m$ conjugate elements $g_i$ of order $d$ such that $\Phi(g_i)$ is almost cyclic, then $\dim \Phi = n \leq m(d-1)$. Therefore, for a fixed $d$, the smaller $m$ is, the smaller will be the degree and hence the number of the representations to be examined when searching for elements of order $d$ of $G$ represented by almost cyclic matrices. For example, it will turn out that, for every simple sporadic group $G$ different from $M$ and for every conjugacy class $C$ of elements of $G$ of prime-power order $d>4$, two suitable elements of $C$ are enough to generate $G$. That is, we can choose $m=2$ in the bound given above. This drastically reduces the computations necessary to prove the results stated in Section \ref{repres}. 

In view of the above considerations, we need to exploit results on the generation of a group $G$ by conjugates. Furthermore, whenever some necessary data on maximal subgroups are missing in GAP (as in the case of some covering groups), or the maximal subgroups of the group $G$ are not completely known (as in the case of the Monster), it will also be useful to know the order of the product  of certain pairs of conjugate elements. This makes all the more convenient a systematic use of the `structure constants method' (as applied, e.g., in \cite{GM}), though clearly the information obtained in this way is generally more precise than strictly required for our purposes in most cases. 
So, we now recall, for the reader's sake, the basics of the `structure constants method'.

Given a finite group $G$, let $C_1,\ldots, C_k$ be $k\geq 3$ (not-necessarily distinct) conjugacy classes of $G$.
Denote by $\Delta_G=\Delta_G(C_1,\ldots,C_k)$ the number of distinct $k$-tuples $(g_1,\ldots,g_k)$, where $g_i \in C_i$ ($1 \leq i \leq{k-1}$), $g_k$ is a fixed element of the class $C_k$, and $g_1g_2\cdots g_{k-1}=g_k$.
This structure constant can be computed using the (complex) character table. Namely, it is given by the formula
$$\Delta_G(C_1,\ldots,C_k)=\frac{|C_1|\cdots |C_{k-1}|}{|G|}\cdot \sum_{i=1}^r\frac{\chi_i(g_1)\chi_i(g_2)\cdots\chi_i(g_{k-1})\overline{\chi_i(g_k)} }{(\chi_i(1))^{k-2}},$$
where $\chi_1,\ldots,\chi_r$ are the irreducible complex characters of $G$.

Next, for a fixed $g_k \in C_k$ denote by $\Delta^*_G(C_1,\ldots,C_k)$ the number of distinct $k$-tuples $(g_1,\ldots,g_k)$ such that $g_i\in C_i$ ($1 \leq i \leq{k-1}$), $g_1\cdots g_{k-1}=g_k$, and  $G=\langle g_1,\ldots,g_{k-1}\rangle$. If $\Delta^*_G(C_1,\ldots,C_k)>0$, the group $G$ is said to be $(C_1,\ldots,C_k)$-generated. For our purposes, we aim to find the minimal $k$ for which $\Delta^*_G(C_1,\ldots,C_k)$ is positive for  certain classes $C_1,\ldots,C_k$ of elements of a given order.

To this end, let $H$ be a maximal subgroup of $G$ containing a fixed element $g_k\in C_k$, and denote by $\Sigma_H(C_1,\ldots,C_k)$ the number of distinct $(k-1)$-tuples $(g_1,\ldots,g_{k-1})\in C_1\times \ldots\times C_{k-1}$ such that $g_1\cdots g_{k-1}=g_k$ and $\langle g_1,\ldots, g_{k-1}\rangle \leq H$. The value of $\Sigma_H(C_1,\ldots,C_k)$ can be obtained as the sum of the structure constants $\Delta_H(c_1,\ldots,c_k)$ of $H$ for all the $H$-conjugacy classes $c_1,\ldots,c_k$ such that $c_i\subseteq H\cap C_i$. 
\medskip

Now, the following holds:

\begin{lemma}[e.g. see \cite{GM}]
Let $G$ be a finite group and let $H$ a subgroup of $G$ containing a fixed element $x$. Denote by $h(x,H)$ the number of the distinct conjugates of $H$ containing $x$. If  $(|x|,|N_G(H):H|)=1$, then
$$h(x,H)=\sum_{i=1}^s\frac{|C_G(x)|}{|C_{N_G(H)}(x_i)|},$$
where $x_1,\ldots,x_s$ are representatives of the $N_G(H)$-conjugacy classes fused to the $G$-class of $x$.
\end{lemma}

As a consequence, we obtain an useful lower bound for $\Delta_G^*(C_1,\ldots,C_k)$. Namely:

$$\Delta_G^*(C_1,\ldots,C_k)\geq \Theta_G(C_1,\ldots,C_k),$$
where 
$$ \Theta_G(C_1,\ldots,C_k)= \Delta_G(C_1,\ldots,C_k) -\sum h(g_k,H)\Sigma_H(C_1,\ldots,C_k),$$ 
$g_k$ is a representative of the class $C_k$, and the sum is taken over the representatives $H$ of the $G$-classes of maximal subgroups of $G$ containing elements of all the classes $C_1,\ldots, C_k$.  

Unless $G=M$, $\Theta_G=\Theta_G(C_1,\ldots,C_k)$ can be computed using the GAP routines. Thus, if $\Theta_G>0$, certainly $G$ is $(C_1,\ldots,C_k)$-generated. In particular,  in the case when $C_1= \cdots =C_{k-1}=C$, this tells us that $G$ can be generated by $k-1$ elements suitably chosen from the class $C$.

Furthermore, if $\Theta\leq 0$, in some cases one can prove that the group $G$ is not $(C_1,\ldots,C_k)$-generated (in particular, that $k-1$ is actually the minimum number of elements from a given class $C$ necessary to generate $G$), with the help of the following Lemma (of which we give the straightforward proof for the sake of clarity):

\begin{lemma}\emph{ (cf. \cite{CW})}\label{center}
Let $G$ be a finite centerless group. If $\Delta^*_G(C_1,\ldots,C_k)>0$, then $\Delta^*_G(C_1,\ldots,C_k)\geq |C_G(g_k)|$, for any $g_k\in C_k$.
\end{lemma}
\begin{proof}
As $\Delta^*_G(C_1,\ldots,C_k)>0$, for any fixed element $g_k\in C_k$ there exists at least one $(k-1)$-tuple $(g_1,\ldots,g_{k-1})$ such that 
$$(\ast)\qquad g_i\in C_i,\quad  g_1\cdots g_{k-1}=g_k \,\textrm{ and  }\,G=\langle g_1,\ldots,g_{k-1}\rangle.$$ Let $x \in C_G(g_k)$. Then
$$x(g_1\cdots g_{k-1})x^{-1}=(xg_1x^{-1})\cdots(xg_{k-1}x^{-1})=xg_kx^{-1}=g_k.$$
Thus, the $(k-1)$-tuple $(xg_1x^{-1},\ldots,xg_{k-1} x^{-1})$ also satisfies $(\ast)$.
Furthermore, if $x_1,x_2$ are distinct elements of $C_G(g_k)$, then the $(k-1)$-tuples  $(x_1g_1x_1^{-1},\ldots,x_1g_{k-1} x_1^{-1})$ and $(x_2g_1x_2^{-1},\ldots,x_2g_{k-1} x_2^{-1})$ are also distinct, since $Z(G)=\{1\}$. This implies that there are at least $|C_G(g_k)|$ $(k-1)$-tuples  $(g_1,\ldots,g_{k-1})$ satisying $(\ast)$. That is, $\Delta_G^\ast(C_1,\ldots,C_k)\geq |C_G(g_k)|$.
\end{proof}

So, obviously, if $\Delta_G(C_1,\dots,C_k)<|C_G(g_k)|$ for $g_k\in C_k$, the previous Lemma tells us that $G$ cannot be $(C_1,\ldots,C_k)$-generated.\\


Recall that a non-scalar $g \in GL(V)$ is called pseudoreflection if $g$ acts scalarly on a hyperplane of $V$. Observe that the matrix of a pseudoreflection is almost cyclic. We will use the following result:

\begin{lemma}\label{wag}
Let  $G< GL(n,F)$ be a finite irreducible linear group generated by pseudoreflections. Then $G$ cannot be a sporadic simple group.
 \end{lemma}
 
 \begin{proof}
 The statement follows immediately from the classification theorems due to A.E Zalesski and  V.N. Serezhkin (see \cite{ZS}) and A.O. Wagner (see \cite{Wag1, Wag2}).
 \end{proof}

Finally, for the reader's convenience, we quote the following:

\begin{propo} [\cite{Zal}\label{Zal}]
Let $G$ be a quasi-simple finite sporadic group and let $Z(G)$ be its center. Let $\ell$ be a prime and $P$ be a Sylow $\ell$-subgroup of $G$. Let $K$ be a splitting field for $G$ of characteristic $\ell$ and let $M$ be a faithful irreducible $KG$-module. Suppose that $P$ is cyclic and $M|_P$ contains no submodule isomorphic to the regular $KP$-module. Then $\ell> 2$ and one of the following holds:
\begin{enumerate}
\item $G=M_{11}$, $|P|=11$ and $\dim M=9$ or $10$;
\item $G=M_{23}$, $|P|=23$ and $\dim M=21$;
\item $G/Z(G)=M_{12}$ or $M_{22}$, $|P|=11$, $\dim M=10$ and $|Z(G)|=2$;
\item $G/Z(G)=Suz$, $|P|=7$ or $13$, $\dim M=12$ and $|Z(G)|=6$;
\item $G/Z(G)=J_3$, $|P|=19$, $\dim M=18$ and $|Z(G)|=3$;
\item $G/Z(G)=Ru$, $|P|=29$, $\dim M=28$ and $|Z(G)|=2$;
\item $G/Z(G)=J_2$, $|P|=7$, $\dim M=6$ and $|Z(G)|=2$;
\item $G/Z(G)=Co_1$, $|P|=13$ and $\dim M=24$;
\item $G=J_1$, $|P|=11$ and $\dim M=7$.
\end{enumerate}
Conversely, in all these cases $M|_P$ does not contain the regular $KP$-submodule. So $\dim M< |P|$, except for the following cases, where $|P|=\ell$  and $\dim M =2(\ell-1)$:
\begin{itemize}
\item[(\rm{i})] $\ell=7$, $G/Z(G)=Suz$;
\item[(\rm{ii})] $\ell=13$, $G/Z(G)=Co_1$.
\end{itemize}
Finally, if $g$ is a generator of $P$, the degree of the minimum polynomial of $g|_M$ equals $\dim M$ (and hence $g|_M$ is cyclic), except for the cases $(\rm{i})$ and $(\rm{ii})$, where the degree is $(\dim M)/2$.
\end{propo}

An immediate consequence of the above Proposition is the following, which will be useful in the sequel:

\begin{corol}\label{l+1}
Let $G$ be a quasi-simple finite sporadic group and let $F$ be an algebraically closed field of positive characteristic $\ell$. Let $\Phi$ be a faithful irreducible representation of $G$ over $F$. Suppose that $P= \lan g \ran$ is a Sylow $\ell$-subgroup of $G$ of order $\ell$. The following holds:
 
\begin{enumerate}
\item if $\dim \Phi=\ell$, then $\Phi(g)$ is cyclic;
\item if $\dim \Phi=\ell+1$, then $\Phi(g)$ is almost cyclic.
\end{enumerate}
\end{corol}

\begin{proof}
If $G/Z(G)\in \{Suz, Co_1\}$, then $\dim \Phi>\ell+1$ (e.g. see \cite{J, HM}). By Proposition \ref{Zal}, it follows that $\Phi|_{P}$ contains as a constituent the regular representation of $P$. The statement follows by degree reasons.
\end{proof}

\section{The sporadic groups: generation by conjugates}\label{secgen}

In this section we describe the results that we have obtained on the generation by sets of conjugate elements of a sporadic simple group $G$, exploiting the machinery introduced in Section \ref{mach}. We have made a systematic use of GAP in order to analyze  the generation of $G$ by conjugates via the `structure constants method'.  It turns out that this gives, for each element $g$ of order $n>1$, a set of conjugates of $g$ generating $G$, which is in most cases of minimal size, though not always. In the latter case, a generating set of minimal size has been obtained by direct computation, using either GAP or MAGMA (\cite{MGM}), except when the group involved is too large. More precisely, the results obtained are the following:
\medskip

\subsection{}
 Let $nX$ denote a conjugacy class of $G$ consisting of elements of order $n>1$.  Computing via the GAP routines $\Theta_G$, and thus obtaining a lower bound for $\Delta^\ast_G$, we get the following:
 
\begin{itemize}
\item $G=M_{11}$. Then $G$ is  $(nX,nX,11a)$-generated for $nX\neq 2a,3a$,  while it is $(3a,3a,8a)$-generated and $(2a,2a,2a,11a)$-generated;

\item $G=M_{12}$. Then $G$ is  $(nX,nX,11a)$-generated for $nX\neq 2a,2b, 3a$, while it is $(3a,3a,6a)$-generated, $(2a,2a,2a,11a)$-generated and $(2b,2b,2b,6a)$-generated;

\item $G=J_1$. Then $G$ is  $(nX,nX,19a)$-generated for $nX\neq 2a$, and $(2a,2a,2a,19a)$-generated;

\item $G=M_{22}$. Then $G$ is  $(nX,nX,11a)$-generated for $nX\neq 2a$, and $(2a,2a,2a,11a)$-generated;

\item $G=J_2$. Then $G$ is  $(nX,nX,7a)$-generated for $nX\neq 2a,2b,3a,4a$, while it is $(4a,4a,5c)$-generated, $(nX,nX,nX,7a)$-generated for $nX=2b,3a$, and $(2a,2a,2a,2a,7a)$-generated;

\item $G=M_{23}$. Then $G$ is  $(nX,nX,23a)$-generated for $nX\neq 2a$, and $(2a,2a,2a,23a)$-generated;

\item $G=HS$. Then $G$ is  $(nX,nX,11a)$-generated for $nX\neq 2a,2b,4a$, and $(nX,nX,nX,11a)$-generated for $nX=2a,2b,4a$;

\item $G=J_{3}$. Then $G$ is  $(nX,nX,19a)$-generated for $nX\neq 2a$, and $(2a,2a,2a,19a)$-generated;

\item $G=M_{24}$. Then $G$ is  $(nX,nX,23a)$-generated for $nX\neq 2a,2b$, and $(nX,nX,nX,23a)$-generated for $nX=2a,2b$;

\item $G=McL$. Then $G$ is $(nX,nX,11a)$-generated for $nX\neq 2a,3a$, and $(nX,nX,nX,11a)$ for $nX=2a,3a$;

\item $G=He$. Then $G$ is $(nX,nX,17a)$-generated for $nX\neq 2a,2b, 3a$,  while it is $(3a,3a,8a)$-generated and $(nX,nX,nX,17a)$ for $nX=2a,2b$;

\item $G=Ru$. Then $G$ is  $(nX,nX,29a)$-generated for $nX\neq 2a,2b$, and $(nX,nX,nX,29a)$-generated for $nX=2a,2b$;

\item $G=Suz$. Then $G$ is $(nX,nX,13a)$-generated for $nX\neq 2a,2b, 3a$, $(nX,nX,nX,13a)$-generated for $nX= 2a,2b$, and $(3a,3a,3a,3a,13a)$-generated;

\item $G=O'N$. Then $G$ is $(nX,nX,31a)$-generated for $nX\neq 2a$, and $(2a,2a,2a,31a)$-generated;

\item $G=Co_3$. Then $G$ is  $(nX,nX,23a)$-generated for $nX\neq 2a,2b, 3a$, while it is $(3a,3a,15a)$-generated and $(nX,nX,nX,23a)$-generated for $nX=2a,2b$;

\item $G=Co_2$. Then $G$ is  $(nX,nX,23a)$-generated for $nX\neq 2a,2b,2c, 4a$, while it is $(4a,4a,10a)$-generated, $(nX,nX,nX,23a)$-generated for $nX=2b,2c$, and $(2a,2a,2a,2a,23a)$-generated;

\item $G=Fi_{22}$. Then $G$ is $(nX,nX,13a)$-generated for $nX\neq 2a,2b, 2c,3a,3b$, while it is $(nX,nX,nX,13a)$-generated for $nX= 2b,2c,3a, 3b$ and $(2a,2a,2a,2a,2a,2a,13a)$-generated;

\item $G=HN$. Then $G$ is $(nX,nX,19a)$-generated for $nX\neq 2a,2b$, and $(nX,nX,nX,19a)$-generated for $nX\neq 2a,2b$;

\item $G=Ly$. Then $G$ is  $(nX,nX,67a)$-generated for $nX\neq 2a, 3a$, and $(nX,nX,nX,67a)$-generated for $nX=2a,3a$;

\item $G=Th$. Then $G$ is $(nX,nX,31a)$-generated for $nX\neq 2a$ and $(2a,2a,2a,31a)$-generated;

\item $G=Fi_{23}$. Then $G$ is $(nX,nX,23a)$-generated for $nX\neq 2a,2b, 2c,3a,3b$, $(nX,nX,nX,23a)$-generated for $nX= 2b,2c,3a, 3b$, and $(2a,2a,2a,2a,2a,2a,23a)$-generated;

\item $G=Co_1$. Then $G$ is  $(nX,nX,23a)$-generated for $nX\neq 2a,2b, 2c,3a,3b,4a$, while it is $(3b,3b,26a)$-generated, $(4a,4a,16b)$-generated,  $(nX,nX,nX,23a)$-generated for $nX=2b,2c$,  $(2a,2a,2a,13a)$-generated, and $(3a,3a,3a,10e)$-generated.

\item $G=J_4$. Then $G$ is $(nX,nX,43a)$-generated for $nX\neq 2a,2b$, and $(nX,nX,nX,43a)$-generated for $nX= 2a,2b$.

\item $G=Fi_{24}'$. Then $G$ is $(nX,nX,29a)$-generated for $nX\neq 2a,2b, 3a,3b$ and $(nX,nX,nX,29a)$-generated for $nX= 2a,2b,3a, 3b$.

\item $G=B$. Then $G$ is $(nX,nX,47a)$-generated for $nX\neq 2a,2b, 2c,2d$, $(nX,nX,nX,47a)$-generated for $nX= 2b,2c,2d$, and  $(2a,2a,2a,2a,47a)$-generated.

\end{itemize}

\subsection{}

Let us denote by $\alpha_G(nX)$ the minimum number of elements from a given non-trivial class $nX$ required to generate $G$. Applying Lemma \ref{center}, we obtain the following estimates:
 \begin{itemize}
 \item  $\alpha_G(nX)\geq 3$ for $(G,nX)\in \{(J_2,3a)$, $(McL, 3a)$, $(Suz, 3a)$, $(Fi_{22}, 3a)$, $(Ly,3a)$, $(Fi_{23}, 3a)$, $(Fi_{23}, 3b)$, $(Co_1,3a)$, $(Fi'_{24},3a)$, $(Fi'_{24},3b)\}$;
 \item $\alpha_G(2a)\geq 5$ for $G=Fi_{22}, Fi_{23}$.  
\end{itemize}

This means that for these groups and classes the size of a generating set from the class $nX$ given above in $\mathbf{3.1}$ is the best possible, except possibly for the cases $(Suz, 3a)$ and $(Fi_{22}, 3b)$.
 \medskip

We can also obtain the exact value of $\alpha_G(nX)$ for $(G,nX)\in \{(HS,4a)$,$(J_2,2a)\}$. Namely, using MAGMA and the representations available in \cite{AOL}, we get the following:
\begin{itemize}
\item for $G=HS$, $\alpha_G(4a)=3$ (here we have looked at an irreducible representation of $G$ of degree $20$ over $\F_2$);
\item for $G=J_2$, $\alpha_G(2a)=4$ (here we have looked at an irreducible representation of $G$ of degree $6$ over $\F_2$).
\end{itemize}

\subsection{}

Finally, let us consider the Monster group $M$. This group requires a slightly different approach, since not all the information we need is available in GAP (in fact, the maximal subgroups of this group are not yet completely known).





So, let $G=M$. Observe (e.g., see \cite{W}) that there are no maximal subgroups of $G$ containing both elements of order $59$ and $71$. Computing the structure constants, we obtain the following:

\begin{enumerate}
\item $\Delta_G(nX,nX,59a)> 0$ and $\Delta_G(nX,nX,71a)> 0$ for all the classes $nX\neq 2a,2b$;
\item $\Delta_G(nX,nX,nX,59a)> 0$ and  $\Delta_G(nX,nX,nX,71a)> 0$ for the classes $nX=2a,2b$.
\end{enumerate}





It follows that $G$ can be generated by $3$ conjugates from each class $nX\neq 2a$, $2b$, and  by $5$ conjugates from each of the classes $nX=2a$, $2b$. However,  in \cite{Ward}, it is shown that $3$ suitable conjugates from the class $2b$ can generate $G$. Furthermore, for the class $2a$, a better bound was obtained by Zisser in \cite{Zis}, namely: $\alpha_M(2a)\leq 4$.

\medskip

The results obtained above may be summarized in the following:

\begin{theor}\label{gen}

Let $G$ be a finite sporadic simple group, and let $g$ be a non-identity element of $G$. Denote by $\alpha_G(g)$ the minimum number of conjugates of  $g$ required to generate $G$. Then the following holds:

\begin{enumerate}
\item If $G\neq M$ and $g \in G$ is not an involution, then $\alpha_G(g)=2$  unless:

\begin{itemize}
 \item $(G,g)\in \{(J_2,3a), (HS, 4a), (McL,3a),  (Ly,3a), (Co_1,3a),
 (Fi_{22}, 3a)$, $(Fi_{23}, 3a)$, $(Fi_{23},3b)$, $(Fi_{24}', 3a), (Fi_{24}', 3b) \}$. In these cases $\alpha_G(g) = 3$;
\item $(G,g)= (Fi_{22}, 3b)$, in which case $2\leq \alpha_G(g) \leq 3$;

\item $(G,g)=(Suz,3a)$, in which case $3\leq \alpha_G(g) \leq 4$;

\end{itemize}

\item If $G\neq M$ and $g \in G$ is an involution, then $\alpha_G(g)=3$ unless:
\begin{itemize}
 \item $(G,g)\in \{(J_2, 2a), (Co_2, 2a), (B,2a)\}$, in which case $\alpha_G(g)=4$;
 \item $(G,g)\in \{(Fi_{22}, 2a), (Fi_{23}, 2a)\}$. In these cases $5\leq \alpha_G(g)\leq 6$;
\end{itemize}

\item If $G=M$ and $g \in G$ is not an involution, then $2\leq \alpha_G(g)\leq 3$;
\item If $G=M$ and $g \in G$ is an involution, then $3\leq \alpha_G(g)\leq 4$.

\end{enumerate}

\end{theor}

\section{The covering groups: generation by conjugates}\label{covgen}

The covering groups of the simple sporadic groups can be dealt with using the same machinery exploited above. Likewise, the notation (notably for conjugacy classes) is the one fixed in the Introduction, following \cite{AOL} and \cite{GAP}.  

For the reader's sake, the following elementary observation seems to be in order.

Suppose that $G$ is a covering group of the simple group $H$ (that is, $G$ is quasi-simple with $G/Z(G)=H$). Then, if $S$ is any generating set for $H$, its preimage in $G$ via the natural map is clearly a generating set for $G$. Obviously, if $S$ consists of elements of the same order, then its preimage in $G$ consists of elements of the same order modulo $Z(G)$.

In view of this, by taking preimages we can transfer the information obtained in the previous section on the generation by conjugates of the sporadic simple groups to their covering groups. The overall results are summarized in the following:

\begin{theor}\label{gencov}
Let $G$ be a covering group of a finite simple sporadic group and let $g$ be a non-central element of $G$. Denote by $\alpha_G(g)$ the minimal number of conjugates of $g$ required to generate $G$, and by $\bar g$ the image of $g$ in $\bar G=G/Z(G)$.
Then the following holds:

\begin{enumerate}

\item If $\bar g$ is not an involution, then $\alpha_G(g)=2$,  unless:

\begin{itemize}

\item $(G, g)\in \{(2^.J_2,3a),(2^. J_2, 6a)$, $ (2^.HS, 4b),
(2^.HS, 4c), (3^. McL,3c)$,
$ (3^.McL,3d)$,  $(3^.McL,3e)$,
 $(2^.Fi_{22}, 3a)$, $(2^. Fi_{22}, 6a)$,
 $(3^. Fi_{22}, 3c)$,$(3^. Fi_{22}, 3d)$,
$(3^. Fi_{22}, 3e)$, $(6^.Fi_{22}, 3c)$,
$(6^. Fi_{22}, 3d)$, $(6^. Fi_{22}, 3e)$,
$(6^. Fi_{22}, 6m)$, $(6^. Fi_{22}, 6n)$,
$(6^.Fi_{22}, 6o)$,  $(2^.Co_1,3a)$,
 $(2^.Co_1,6a)$, $(3^.Fi_{24}', 3c)$,
 $(3^. Fi_{24}', 3d) \}$, $(3^.Fi_{24}', 3e) \}$, $(3^. Fi_{24}', 3f) \}$.
In these cases $\alpha_G(g) = 3$;

\item $(G, g)\in \{
(2^.Fi_{22}, 3b), (2^. Fi_{22}, 6b)$,
$(3^.Fi_{22}, 3f), (6^.Fi_{22}, 3f)$,
$(6^. Fi_{22}, 6p)\}$. In these cases $2\leq \alpha_G(g) \leq 3$;

\item $(G, g)\in \{
(2^. Suz,3a),
(2^. Suz,6a),
(3^. Suz,3c),
(3^. Suz,3d),
 (3^. Suz,3e),
(6^. Suz,3c)$, $(6^. Suz,3d),
(6^. Suz,3e),
(6^. Suz,6g),
(6^. Suz,6h),
(6^. Suz,6i)\}$. In these cases $3\leq \alpha_G(g) \leq 4$;

\end{itemize}

\item  If $\bar g$ is an involution, then $\alpha_G(g)=3$, except for the following cases:
\begin{itemize}
 \item If $(G, g)\in \{(2^. J_2, 2b), (2^. J_2, 2c),(2^. B, 2b)\} $, then $\alpha_G(g)=4$;

 \item
   If $(G, g)\in \{
(2^. Fi_{22}, 2b),
(2^. Fi_{22}, 2c),
(3^. Fi_{22}, 2a),
(3^. Fi_{22}, 6a),
(3^. Fi_{22}, 6b)$,
$(6^. Fi_{22}, 2b)$,
$(6^. Fi_{22}, 2c),
(6^. Fi_{22}, 6c),
(6^. Fi_{22}, 6d),
(6^. Fi_{22}, 6e),
(6^. Fi_{22}, 6f)\} $, then $5\leq \alpha_G(g)\leq 6$.
\end{itemize}



\end{enumerate}

\end{theor}

\section{Cyclic and almost cyclic elements in the representations of finite sporadic groups}\label{spor}

In this section we determine the occurrence of cyclic and almost cyclic elements in the representations of the finite sporadic simple groups.  For their relevance as well as for technical reasons, we have confined our analysis to the case of elements of prime-power order. The results will be  summarized in Theorem \ref{cyclic} and Theorem \ref{almost}  (Section \ref{repres}).
\smallskip

To simplify the notation, if $\Phi: G\rightarrow GL(n,F)$ is a faithful irreducible representation,  we will identify $G$ with $\Phi(G)$ (and $\Phi(g)$ with $g$). Moreover, when we say below that an element of $G$ is almost cyclic, we mean that the element is almost cyclic but not cyclic, and when we say that an element is not almost cyclic, we mean that it is neither cyclic nor almost cyclic. Finally, we denote by $d$ a prime-power integer, and we always assume $d>2$. Indeed, observe that, in view of Lemma \ref{wag}, no involution of a sporadic simple group can be represented by a cyclic matrix, and hence we will disregard completely generating sets of involutions in our analysis.
\smallskip

 In order to apply Lemma \ref{2.1}, we fully exploit the results obtained in Section \ref{secgen} on the generation of $G$ by conjugates. 
Next, we refer to the paper of Jansen (\cite{J}), giving the minimal degree of the faithful irreducible representations $\Phi$ of $G$, as well as to the work of Hiss and Malle (\cite{HM}) on the low-dimensional representations of quasi-simple groups. By Lemma \ref{2.1}, we must have $\dim \Phi\leq \alpha_G(g)(|g|-1)$. If this bound is not met by any $\Phi$, then we are done: $\Phi(g)$ cannot be neither cyclic nor almost cyclic. Otherwise, the list of representations meeting the bound is usually small: if $\ell$ does not divide the order of $g$, and the relevant Brauer character tables are known, we get the desired answers using GAP; otherwise, we make use of MAGMA,  applying it to the relevant representations as provided by the Atlas on line (\cite{AOL}).
\medskip

The results obtained are as follows:

\subsection {$G$ = $McL$, $He$, $Suz$, $O'N$, $J_4$, $HN$, $Th$, $Fi_{22}$, $Fi_{23}$,  $Fi_{24}', B$}


\medskip
By Theorem \ref{gen}, it follows from Lemma \ref{2.1} and  \cite{J} that for all the listed groups $G$ no element $g$ of prime-power order $d$ can be almost cyclic.

\subsection{$G =M_{11}$}

In view of \cite{J} and Lemma \ref{2.1}, we are left to examine the classes $4a$ and $5a$  only for $\ell= 3$, and the classes $8a$, $8b$, $11a$ and $11b$ for every  $\ell$. Since the Brauer character tables are known for any characteristic, using the GAP routines we can answer completely the case when $\ell$ does not divide $d$. We obtain that cyclic or almost cyclic elements occur exactly as listed in the following table:

\begin{center}
\begin{tabular}{c|c|c|c}
$\ell$  & $\dim \Phi$ & $dX$ & type \\ \hline\hline
$\ell \nmid |G|$  & $10$ & $11a,b$  & cyclic\\
{} & $11$ & $11a,b$ & cyclic \\ \hline
$2$ & $10$ & $11a,b$ & cyclic \\ \hline
$3$ & $5$ & $4a$ & almost cyclic \\
   &   $5$  &  $5a$;  $8a,b$; $11a,b$ & cyclic \\
  &  $10$ & $11a,b$ & cyclic  \\ \hline 
$5$ & $10$ & $11a,b$ & cyclic \\
   &  $11$  &  $11a,b$ & cyclic \\ \hline
$11$ & $9$ & $8a,b$ & almost cyclic \\
\end{tabular} 
\end{center}

So, we are left to examine the classes $8a,8b$ when $\ell=2$ and the classes $11a,11b$ when $\ell=11$. Note that $(8a)^5 \in 8b$ and $(11a)^2\in 11b$, and therefore for our purposes it is irrelevant whether an element $g$ of order $8$ (resp. $11$) belongs the class $8a$ or $8b$ (resp. $11a$ or $11b$). Denoting by $a$ and $b$ the 'standard generators' of $M_{11}$, of order respectively $2$ and $4$, given in \cite{AOL}, the following holds:

(i) If $\ell=2$, by \cite{HM} the bound given by  Lemma \ref{2.1} is only met by a representation $\Phi$ of degree $10$. Pick $g=bab^2 (ab)^3 ba$. Then $g$ has order $8$, and its invariant factors (computed using MAGMA) are $\{(x-1)^2$, $(x-1)^8\}$. So $g$ is not almost cyclic.

(ii) If $\ell=11$, by \cite{HM} $\Phi$ must have degree $9$, $10$, $11$ or $16$. Pick $g=ab$. Then $g$ has order $11$. If $\dim \Phi=9$ or $10$, by Proposition \ref{Zal} $g$  is cyclic. If $\dim \Phi=11$, $g$ is cyclic by Corollary \ref{l+1}. Finally, if $\dim \Phi=16$, the invariant factors of $g$ are $\{(x-1)^5$, $(x-1)^{11}\}$. So $g$ is not almost cyclic.

\subsection{$G=M_{12}$}

In view of \cite{J} and Lemma \ref{2.1}, we are left to examine the classes $8a$, $8b$, $11a$ and $11b$, for every $\ell$. Since the Brauer character tables are known for any characteristic, using the GAP routines we obtain that, whenever $\ell$ does not divide $d$, cyclic or almost cyclic elements occur exactly according to the following table:

\begin{center}
\begin{tabular}{c|c|c|c}
$\ell$  & $\dim \Phi$ & $dX$ & type \\ \hline\hline
$\ell\nmid |G|$ & $11$ & $11a,b$ & cyclic \\ \hline
$2$ & $10$ & $11a,b$ & cyclic \\ \hline
$3$ & $10$ & $11a,b$ & cyclic \\ \hline
$5$ & $11$ & $11a,b$ & cyclic\\ 
\end{tabular} 
\end{center}

So, we are left to examine the classes $8a,8b$ when $\ell=2$  and the classes $11a, 11b$ when $\ell=11$. Note that $(11a)^2\in 11b$, and therefore for our purposes it is irrelevant whether an element $g$ of order $11$ belongs the class $11a$ or $11b$. Denoting by $a$ and $b$ the 'standard generators' of $M_{12}$, of order respectively $2$ and $3$, given in \cite{AOL}, the following holds:

(i) If $\ell=2$, by \cite{HM} $\Phi$ must have degree $10$. Constructing this representation using MAGMA, we obtain that, for $g$ in both classes $8a$ and $8b$, the invariant factors are $\{(x-1)^2$, $(x-1) ^8\}$. Thus $g$ is not almost cyclic.

(ii) If $\ell=11$, by \cite{HM} $\Phi$ must have either degree $11$ (there are two such representations) or degree $16$.  Pick $g=ab$. Then $g$ has order $11$. If $\dim \Phi=11$, then $g$ is cyclic by Corollary \ref{l+1}. If $\dim \Phi=16$, the invariant factors of $g$ are $\{(x-1)^5$, $(x-1)^{11}\}$. So $g$ is not almost cyclic.

\subsection{$G=J_1$}
In view of \cite{J} and Lemma \ref{2.1}, we only need to examine the classes $11a$, $19a$, $19b$ and $19c$ when $\ell=2$, the classes $19a$, $19b$ and $19c$ when $\ell=7,19$, and the classes $5a$, $5b$, $7a$, $11a$, $19a$, $19b$ and $19c$ when $\ell=11$.
The Brauer character tables being known for any characteristic, using the GAP routines  we obtain that, whenever $\ell$ does not divide $d$, cyclic or almost cyclic elements occur exactly according to the following table:

\begin{center}
\begin{tabular}{c|c|c|c}
$\ell$  & $\dim \Phi$ & $dX$ & type \\ \hline\hline
$2$ & $20$ & $19a,b,c$ & almost cyclic \\ \hline
$11$ & $7$ & $7a$; $19a,b,c$ & cyclic\\ 
{} & $14$ & $19a,b,c$ & almost cyclic
\end{tabular} 
\end{center}

So, we are left to examine the class $11a$ for $\ell=11$ and the class $19a$ for $\ell=19$ (note that $(19a)^2\in 19b$ and $(19b)^2\in 19c$).
Denoting by $a$ and $b$ the 'standard generators' of $M_{12}$, of order respectively $2$ and $3$, given in \cite{AOL}, the following holds:

(i) If $\ell=11$, by \cite{HM} $\Phi$ must have either degree $7$ or degree $14$. Let us pick $g=baba b^2 abab^2 ab^2 (ab)^3 ab^2 ab$. Then $g$ has order $11$. If $\dim \Phi=7$, then $g$ is cyclic by Proposition \ref{Zal}. If $\dim \Phi=14$, the invariant factors of $g$ are $\{(x-1)^3$, $(x-1)^{11}\}$, and hence $g$ is not almost cyclic.

(ii) If $\ell=19$, by \cite{HM} $\Phi$ must have either degree $22$ or degree $34$.  Pick $g=abab^2$. Then $g$ has order $19$. If $\dim \Phi=22$,  the invariant factors of $g$ are $\{(x-1)^3$, $(x-1)^{19}\}$.  If $\dim \Phi=34$, the invariant factors of $g$ are $\{(x-1)^{15}$, $(x-1)^{19}\}$.  Thus, in both cases, $g$ is not almost cyclic.

\subsection{$G=M_{22}$} 
In view of \cite{J} and Lemma \ref{2.1}, we only need to examine the classes $7a$, $7b$, $8a$, $11a$ and $11b$ when $\ell=2$, and the classes $11a$ and $11b$ when $\ell=11$.  Also note that $(11a)^2\in 11b$.

Denoting by $a$ and $b$ the 'standard generators' of $M_{12}$, of order respectively $2$ and $4$, given in \cite{AOL}, the following holds:

(i) If $\ell=2$, by \cite{HM} $\Phi$ must have degree $10$ (there are two such representations). By inspection of the Brauer character tables, we see that if $g$ has order $11$, then $g$ is cyclic, whereas if $g$ has order $7$, then $g$ is not almost cyclic.
Next, pick $g= bab^2abab^2ab^2ab^2aba$. Then $g$ has order $8$ and, using MAGMA, we get that the invariant factors of $g$ are $\{(x-1)^2$, $(x-1)^8\}$. Thus $g$ is not almost cyclic.

(ii) If $\ell=11$, by \cite{HM} $\Phi$ must have degree $20$. Let $g=ab$. Then $g$ has order $11$. In this case, the invariant factors of $g$ are $\{(x-1)^9$, $(x-1)^{11}\}$. Thus $g$ is not almost cyclic.

\subsection{$G=J_2$}
In view of \cite{J} and Lemma \ref{2.1}, we only need to examine the class $8a$ when $\ell\neq 2$ and the classes $3a$, $4a$, $5a$, $5b$, $5c$, $5d$, $7a$ and $8a$ when $\ell=2$. 

Denoting by $a$ and $b$ the 'standard generators' of $M_{12}$, of order respectively $2$ and $3$, given in \cite{AOL}, the following holds:

(i) By inspection of the Brauer character tables, whenever $\ell$ does not divide $d$, there is only one instance in which a cyclic or almost cyclic element can occur, namely the following:

\begin{center}
\begin{tabular}{c|c|c|c}
$\ell$  & $\dim \Phi$ & $dX$ & type \\ \hline\hline
$2$ & $6$ & $7a$ & cyclic \\
\end{tabular} 
\end{center}

(ii) If $\ell=2$, by \cite{HM} $\Phi$ must have either degree $6$ or degree $14$.  Moreover, $g_4=(ab^2ab)^3$ has order $4$ and $g_8= (abab^2) ^2 (ab)^3bab^2ab$ has order $8$.

Let $\dim \Phi=6$ (there are two such representations). Then the invariant factors of $g_4$ are $\{(x-1)^3$, $(x-1)^3\}$, and hence $g_4$ is not almost cyclic. On the other hand, the minimal and characteristic polynomial of $g_8$ coincide. So $g_8$ is cyclic.

Let $\dim \Phi=14$ (there are two such representations). In view of Lemma \ref{2.1}, we only need to deal with $g_8$. As the invariant factors of $g_8$ are $\{(x-1)^6$, $(x-1)^8\}$, $g_8$ is not almost cyclic.

\subsection{$G=M_{23}$}
In view of \cite{J} and Lemma \ref{2.1}, we only need to examine the classes $23a$ and $23b$ when $\ell\neq 2$, and the classes $7a$, $7b$, $8a$, $11a$, $11b$, $23a$ and $23b$, when $\ell=2$. Since the Brauer character tables are known for any characteristic, using the GAP routines we obtain that, whenever $\ell$ does not divide $d$, cyclic or almost cyclic elements occur exactly according to the following table:

\begin{center}
\begin{tabular}{c|c|c|c}
$\ell$  & $\dim \Phi$ & $dX$ & type \\ \hline\hline
$\ell \nmid |G|$ & $22$ & $23a,b$ & cyclic \\ \hline
$2$ & $11$ & $11a,b$; $23a,b$ &  cyclic \\ \hline
$3$ & $22$ & $23a,b$ & cyclic \\ \hline
$5$ & $22$ & $23a,b$ & cyclic \\ \hline
$7$ & $22$ & $23a,b$ & cyclic \\ \hline
$11$ & $22$ & $23a,b$ & cyclic \\
\end{tabular} 
\end{center}

Let us denote by $a$ and $b$ the 'standard generators' of $M_{23}$, of order respectively $2$ and $4$, given in \cite{AOL}.

Let $\ell=2$. By \cite{HM} $\Phi$ must have degree $11$. Pick $g= (ab)^4 b (ab)^2 bab^2$. Then $g$ has order $8$, and the invariant factors of $g$ are $\{(x-1)^3$, $(x-1)^8\}$. So $g$ is not almost cyclic.

Let $\ell=23$. By \cite{HM} $\Phi$ must have  degree $21$. In this case, an element of order $23$ is cyclic by Proposition \ref{Zal}.

\subsection{$G=HS$} 
Arguing as above, we have only to examine the classes $11a$ and $11b$ when $\ell=2$. These elements turn out not to be almost cyclic.

\subsection{$G=J_3$}
The only classes to be examined are $17a$, $17b$, $19a$ and $19b$ when $\ell=3$. By \cite{HM} $\Phi$ must have degree $18$ (there are two such representations). Using as above the GAP routines, it turns out that the elements of order $17$ are almost cyclic, while the elements of order $19$ are cyclic.

\subsection{$G=M_{24}$}
The only classes to be examined are the classes $23a$ and $23b$  when $\ell\neq 2$, and the classes $7a$, $7b$ $8a$, $11a$, $23a$ and $23b$ when $\ell=2$. Since the Brauer character tables are known for any characteristic, using  the GAP routines we obtain that, whenever $\ell$ does not divide $d$, cyclic or almost cyclic elements occur exactly according to the following table:

\begin{center}
\begin{tabular}{c|c|c|c}
$\ell$  & $\dim \Phi$ & $dX$ & type \\ \hline\hline
$\ell \nmid |G|$ & $23$ & $23a,b$ & cyclic \\ \hline
$2$ & $11$ & $11a$; $23a,b$ & cyclic \\ \hline
$3$ & $22$ & $23a,b$ & cyclic \\ \hline
$5$ & $23$ & $23a,b$ & cyclic \\ \hline
$7$ & $23$ & $23a,b$ & cyclic \\ \hline
$11$ & $23$ &  $23a,b$ & cyclic \\
\end{tabular} 
\end{center}

Let us denote by $a$ and $b$ the 'standard generators' of $M_{24}$, of order respectively $2$ and $4$, given in \cite{AOL}.

If $\ell=2$, by \cite{HM} $\Phi$ must have degree $11$ (there are two such representations). Pick $g=(ba)^2 (b^2a)^2 b$. Then $g$ has order $8$, and the invariant factors of $g$ are $\{(x-1)^3$, $(x-1)^8\}$. So $g$ is not almost cyclic.

If $\ell=23$, by \cite{HM} $\Phi$ must have degree $23$. In this case, the elements of order $23$ are cyclic by Corollary \ref{l+1}.

\subsection{$G=Ru$}
The only classes to be examined are $16a$, $16b$, $29a$ and $29b$, when $\ell=2$. Note that $(16a)^3\in (16b)$.
By \cite{HM}, $\Phi$ must be of degree $28$.  A GAP computation shows that the elements of order $29$ are cyclic. 
Now, denote by $a$ and $b$ the 'standard generators' of $Ru$, of order respectively $2$ and $4$, given in \cite{AOL}. Pick $g= (ba)^2 b^2 a b^2 (ab)^3 bab (ab^2)^5 (a  b)^2 b (ab) ^2$. Then $g$ has order $16$ , and it is not almost cyclic, since its invariant factors are $\{(x-1)^{12}$, $(x-1)^{16}\}$.

\subsection{$G=Co_3$}
The only classes to be examined are the classes $23a$ and $23b$, for every $\ell$. Since the Brauer character tables are known for any characteristic, using  the GAP routines we obtain that, whenever $\ell$ does not divide $d$, cyclic or almost cyclic elements occur exactly according to the following table:

\begin{center}
\begin{tabular}{c|c|c|c}
$\ell$  & $\dim \Phi$ & $dX$ & type \\ \hline\hline
$\ell \nmid |G|$ & $23$ & $23a,b$ & cyclic \\ \hline
$2$ & $22$ & $23a,b$ & cyclic \\ \hline
$3$ & $22$ & $23a,b$ & cyclic \\ \hline
$5$ & $23$ & $23a,b$ & cyclic \\ \hline
$7$ & $23$ & $23a,b$ & cyclic \\ \hline
$11$ & $23$ & $23a,b$ & cyclic \\ 
\end{tabular} 
\end{center}
 
If $\ell=23$, by \cite{HM} $\Phi$  must have degree $23$. Hence the elements of order $23$ are cyclic by Corollary \ref{l+1}.

\subsection{$G=Co_2$}

The only classes to be examined are the classes $16a$, $16b$, $23a$ and $23b$, for every $\ell$. Since the Brauer character tables are known for any characteristic, using  the GAP routines we obtain that, whenever $\ell$ does not divide $d$, cyclic or almost cyclic elements occur exactly according to the following table:

\begin{center}
\begin{tabular}{c|c|c|c}
$\ell$  & $\dim \Phi$ & $dX$ & type \\ \hline\hline
$\ell\nmid |G|$ & $23$ & $23a,b$ & cyclic \\ \hline
$2$ & $22$ & $23a,b$ & cyclic \\ \hline
$3$ & $23$ & $23a,b$ & cyclic \\ \hline
$5$ & $23$ & $23a,b$ & cyclic \\ \hline
$7$ & $23$ & $23a,b$ & cyclic \\ \hline
$11$ & $23$ & $23a,b$ & cyclic \\
\end{tabular} 
\end{center}

If $\ell=2$, we need to examine the classes $16a$ and $16b$. By \cite{HM}, $\Phi$ must have degree $22$. Using MAGMA, we can check that, for both classes, the invariant factors are $\{(x-1)^8$, $(x-1)^{14}\}$. So, these elements are not almost cyclic.

If $\ell=23$, by \cite{HM} $\Phi$ must have degree $23$, and the elements of order $23$ are cyclic by Corollary \ref{l+1}.

\subsection{$G=Ly$}
$G$ is $(dX,dX,67a)$-generated for all the classes $dX \neq 3a$, while it is  $(3a,3a,3a,67a)$-generated (three being the minimal size of a generating set from the class $3a$). The only classes to be examined are the classes $67a$, $67b$ and $67c$, when $\ell=5$. Observe that $(67a)^2\in 67b$ and $(67a)^7 \in 67c$, and therefore for our purposes it is irrelevant whether an element $g$ of order $67$ belongs to one or another class. Denote by $a$ and $b$ the 'standard generators' of $Ly$, of order respectively $2$ and $5$, given in \cite{AOL}. According to \cite{HM}, $\Phi$ must have degree $111$, and moreover is unique (unpublished work of Lux and Ryba). Let us pick $g=(ab)^3 b$. Then $g$ has order $67$, but it is not almost-cyclic, since its minimum polynomial is $m_g(x)=x^{67}-1$, whereas its characteristic polynomial is $p_g(x)=(x^{67}-1)(x^{22}+x^{20}-x^{18}+2x^{17}-x^{16}-x^{15}+2x^{14}+x^{12}+x^{10}+2x^8-x^7-x^6+2x^5-x^4+x^2+1)( x^{22}-x^{21} + 3x^{20}+2x^{19}-x^{18}+2x^{15}+2x^{14}-x^{12}+3x^{11}-x^{10}+2x^8+2x^7-x^4+2x^3+3x^2-x+1)$.

\subsection{$G=Co_1$}

The only classes to be examined are the classes $13a$, $16a$, $16b$, $23a$ and $23b$, when $\ell=2$. Also, note that $(23a)^5\in 23b$. Let us denote by $a$ and $b$ the 'standard generators' of $Co_1$, of order respectively $2$ and $3$, given in \cite{AOL}. According to \cite{HM}, $\Phi$ must have degree $24$. Moreover, such a $\Phi$ is unique (since a proof of this fact is not available in the literature, we have checked it independently. See Appendix).  Pick $g=(ba)^2 b^2 (ab)^2  (b a)^4 b  (b a)^3$. Then $g$ has order $23$ and it is almost cyclic, since it has minimum polynomial $m_g(x)=x^{23}-1$ and characteristic polynomial $p_g(x)=(x-1)(x^{23}-1)$. Next, let $g= (b( abab^2)^2 ab^2a b(ab^2)^5abab^2(ab)^2)^2$. Then $g$ has order $13$, its minimal polynomial is $m_g(x)=\frac{x^{13}-1}{x-1}$ and its characteristic polynomial is $p_g(x)=(m_g(x))^2$. So $g$ is not almost cyclic.\\
Finally, the elements of order $16$ cannot be almost cyclic. Indeed, assume that $g\in G$ of order $16$ is such that $\Phi(g)$ is almost cyclic. Observe that both classes $16a$ and $16b$ have non-trivial intersection with a maximal subgroup $H$ of type $Co_2$. Since the minimal degree of an irreducible representation of $Co_2$ is $22$, $\Phi_{|H}=2 \Psi_1+\Psi_{22}$, where $\Psi_i$ are irreducible representations of $H$ of degree $i$. This means that, considering $g$ as an element of $Co_2$, $\Psi_{22}(g)$ should be almost cyclic. But we have already proved that this cannot happen. \\


\subsection{$G=M$}


By Theorem \ref{gen}, $G$ can be generated by at most $3$ conjugates from each class $dX$. Since $\dim \Phi\geq 196882$ by  \cite{J}), by Lemma \ref{2.1} no element $g$ of prime-power order $d$ can be almost cyclic.

\section{Cyclic and almost cyclic elements in the representations of the covering groups}\label{covspor}

The covering groups of the simple sporadic groups can be dealt with using the same machinery exploited above. As in the previous section, for technical reasons, we confine our analysis to the case of elements of prime-power order (modulo the centre) which can be represented by cyclic or almost cyclic matrices in faithful irreducible representations. The notation (notably for conjugacy classes) is the one fixed in the Introduction, following \cite{AOL} and \cite{GAP}.  

For the reader's sake, the following elementary observations seem to be in order:


1) An element $x$ of $G$ has prime-power order modulo $Z(G)$ if and only if $x$ is the product of a central element by an element, say $x_1$, of $G$ of prime-power order. Obviously, for any $F$-representation $\Phi$ of $G$, $\Phi(x)$ is cyclic (almost cyclic) if and only if $\Phi(x_1)$ is cyclic (almost cyclic).

2) Let  $\Phi$ be an $F$-representation of $G$ and let  $\tilde{\Phi}$ be the associated projective representation of its simple central quotient $H$ 
(defined by $\tilde{\Phi}(Z(G)x)= \Phi(x)$  for $x \in G$). Suppose that $x$ has order two modulo the centre, and $\Phi(x)$ is cyclic (almost cyclic).  Then $\Phi(x)$, and so also $\tilde{\Phi}(Z(G)x)$, is a pseudoreflection. But this contradicts Lemma \ref{wag}. 

We will fully exploit the results on generation by conjugates obtained in Sections \ref{secgen} and \ref{covgen}.
Furthermore, in view of 1) and 2), we will only have to deal with the conjugacy classes of the covering group $G$ which consist of elements of prime-power order $d$ whose images in $H$ have order greater than two. Therefore, from now on, the notation $dX$ only refers to such classes.

\medskip


We obtain the following results:



\subsection{$G=6^. M_{22}$, $12^. M_{22}$, $2^. HS$,  $3^. McL$, $2^.Fi_{22}$, $6^.Fi_{22}$, $3^.Fi_{24}'$, $2^.B$ }
By Theorem \ref{gencov}, it follows from Lemma \ref{2.1} and  \cite{J} that no $g\in G$ belonging to any of the classes $dX$ can be almost cyclic.

\subsection{$G=2^.M_{12}$} 
$G$ can be generated by two conjugates from any of the classes $dX$.

In view of \cite{J} and Lemma \ref{2.1}, by inspecting the Brauer character tables and using the GAP routines  we obtain that, whenever $\ell$ does not divide $d$, cyclic or almost cyclic elements occur exactly according to the following table:

\begin{center}
\begin{tabular}{c|c|c|c}
$\ell$  & $\dim \Phi$ & $dX$ & type \\ \hline\hline
$\ell \nmid |G|$ & $10$ & $11a,b$ & cyclic \\ 
    & $12$ & $11a,b$ & almost cyclic \\ \hline
$3$ & $6$ & $5a$ & almost cyclic \\
    & $6$ & $8a,b,c,d$; $11a,b$ & cyclic \\
    & $10$ & $11a,b$ & cyclic \\  \hline
$5$ & $10$ & $11a,b$ & cyclic \\ 
    & $12$ & $11a,b$ & almost cyclic \\
\end{tabular} 
\end{center}

Again by  \cite{J} and Lemma \ref{2.1}, we may rule out the class $3a$ for $\ell=3$ and the class $5a$ for $\ell=5$. Thus, we are only left to examine the classes $11a$ and $11b$ for $\ell=11$.



Now, if $\ell=11$, by \cite{HM} $\Phi$ must have either degree $10$ (there are two such representations) or degree $12$ . Let $g$ be an element of order $11$. If $\dim \Phi=10$, by Proposition \ref{Zal} $\Phi(g)$ is cyclic. If $\dim \Phi=12$, then $\Phi(g)$ is almost cyclic by Corollary \ref{l+1}.

\subsection{$G=2^.M_{22}$} 
$G$ can be generated by two conjugates from any of the classes $dX$.

In view of \cite{J} and Lemma \ref{2.1}, by inspecting the Brauer character tables and using the GAP routines  we obtain that, whenever $\ell$ does not divide $d$, cyclic or almost cyclic elements occur exactly according to the following table:

\begin{center} 
\begin{tabular}{c|c|c|c}
$\ell$  & $\dim \Phi$ & $dX$ & type \\ \hline\hline
$\ell \nmid |G|$ & $10$ & $11a,b$ & cyclic \\  \hline
$3$ & $10$ & $11a,b$ & cyclic \\ \hline
$5$ & $10$ & $11a,b$ & cyclic \\ \hline
$7$ & $10$ & $11a,b$ & cyclic \\ 
\end{tabular} 
\end{center}

Thus, we are left to examine the classes $7a$ and $7b$ for $\ell=7$ and the classes $11a$ and $11b$ for $\ell=11$.

Let us denote by $a$ and $b$ the 'standard generators' of $2^.M_{22}$, of order respectively $2$ and $4$, given in \cite{AOL}.

Let $\ell=7$. By \cite{HM} $\Phi$ must have degree $10$. 
Also, note that $(7a)^3\in 7b$. Pick $g=abab^3ab^2abab^3ab^2a$. Then $g$ has order $7$, and its invariant factors are $\{(x-1)^3$, $(x-1)^7\}$. So $g$ is not almost cyclic.

Let $\ell=11$. By \cite{HM} $\Phi$ must have degree $10$ (there are two such representations). If $g \in G$ has order $11$, then $g$ is cyclic by Proposition \ref{Zal}.

\subsection{$G=3^.M_{22}$} 
$G$ can be generated by two conjugates from any of the classes $dX$.

In view of \cite{J} and Lemma \ref{2.1}, by inspecting the Brauer character tables and using the GAP routines  we obtain that, whenever $\ell$ does not divide $d$, cyclic or almost cyclic elements occur exactly according to the following table:

\begin{center} 
\begin{tabular}{c|c|c|c}
$\ell$  & $\dim \Phi$ & $dX$ & type \\ \hline\hline
$2$ & $6$ & $5a$ & almost cyclic \\ 
    & $6$ & $7a,b$; $11a,b$ & cyclic \\ 
\end{tabular} 
\end{center}
  
Again by \cite{J} and Lemma \ref{2.1}, we are left to examine only the classes $4a$, $4b$ and $8a$ for $\ell=2$.
Let us denote by $a$ and $b$ the 'standard generators' of $3^.M_{22}$, of order respectively $2$ and $4$, given in \cite{AOL}.

If $\ell=2$, by \cite{HM} $\Phi$ must have degree $6$ (there are two such representations).  Using MAGMA, we see that the invariant factors of the elements of order $4$ are either $\{(x-1)^3$, $(x-1)^3\}$ or $\{(x-1)^2$, $(x-1)^4\}$. So these elements are not almost cyclic. On the other hand, for an element of order $8$ the minimum polynomial and the characteristic polynomial coincide. So, the element is cyclic.

\subsection{$G=4^.M_{22}$} 
$G$ can be generated by two conjugates from any of the classes $dX$.
 By Lemma \ref{2.1} and \cite{J}, we only  have to examine the classes $11a$ and $11b$ for $\ell=7$. Inspection of the Brauer character tables shows that these elements are not almost cyclic.

\subsection{$G=2^.J_2$} 
$G$ can be generated by two conjugates from any of the classes $dX \neq 3a$;  while it can be generated by three conjugates from the class $dX=3a$.

Taking into account  \cite{J} and Lemma \ref{2.1},  we obtain the following:

(i) Whenever $\ell$ does not divide $d$, the Brauer character tables, via the GAP routines, show that cyclic or almost cyclic elements occur exactly according to the following table:

\begin{center} 
\begin{tabular}{c|c|c|c}
$\ell$  & $\dim \Phi$ & $dX$ & type \\ \hline\hline
$\ell \nmid |G|$ & $6$ & $7a$; $8a,b$ & cyclic\\ \hline
$3$ & $6$ & $7a$; $8a,b$ & cyclic \\ \hline
$5$ & $6$ & $7a$; $8a,b$ & cyclic \\ \hline
$7$ & $6$ & $8a,b$ & cyclic
\end{tabular} 
\end{center}

(ii) We are left to examine the class $3a$ for $\ell=3$, the classes $5a$ and $5b$ for $\ell=5$ and the class $7a$ for $\ell=7$.

If $\ell=3$, by \cite{HM} $\Phi$ must have degree $6$ (there are two such representations). However, using MAGMA, we see that the elements of order $3$ have as invariant factors either $\{(x-1)^2$, $(x-1)^2$, $(x-1)^2\}$ or  $\{(x-1)^3$, $(x-1)^3\}$.

If $\ell=5$, by \cite{HM} $\Phi$ must have degree $6$.  Using MAGMA, we see that the elements of order $5$ have as invariant factors either  $\{(x-1)^3$, $(x-1)^3\}$ or  $\{(x-1)^2$, $(x-1)^4\}$.

If $\ell=7$, by \cite{HM} $\Phi$ must have degree $6$ (there are two such representations). By Proposition \ref{Zal}, the elements of order $7$ are cyclic.

\subsection{$G=2^.Suz$} 
 $G$ can be generated by two conjugates from any of the classes $dX \neq 3a$, and by four conjugates from the class $3a$.

In view of \cite{J} and Lemma \ref{2.1}, we need to examine only the case $\ell=3$.

If $3$ is coprime to $d$, inspection of the Brauer character table produces the following single occurrence: 

\begin{center} 
\begin{tabular}{c|c|c|c}
$\ell$  & $\dim \Phi$ & $dX$ & type \\ \hline\hline
$3$ & $12$ & $13a,b$ & cyclic \\ 
\end{tabular} 
\end{center}
  
Since, by \cite{HM}, $\Phi$ must have degree $12$, we are left to examine only the classes  $9a$ and $9b$ for $\ell=3$. Also, note that $(9a)^2\in 9b$. Let us denote by $a$ and $b$ the 'standard generators' of $2^.Suz$, of order respectively $4$ and $3$, given in \cite{AOL}. Pick $g=(ab)^3 ab^2 ab  (ab^2)^5 ab$. Then $g$ has order $9$, and its invariant factors are $\{(x-1)^4$, $(x-1)^8\}$. So $g$ is not almost cyclic.

\subsection{$G=3^.Suz$} 
$G$ can be generated by two conjugates from any of  the classes $dX \neq 3c$, $3d$, $3e$, and by four conjugates from any of 
 the classes $dX=3c,3d,3e$. In view of \cite{J} and Lemma \ref{2.1}, we need to examine only the case $\ell=2$.

If $2$ is coprime to $d$, inspection of the Brauer character table produces the following occurrences:

\begin{center} 
\begin{tabular}{c|c|c|c}
$\ell$  & $\dim \Phi$ & $dX$ & type \\ \hline\hline
$2$ & $12$ & $11a$ &  almost cyclic \\  
    & $12$ & $13a,b$ & cyclic 
\end{tabular} 
\end{center}

Since, by \cite{HM}, $\Phi$ must have degree $12$ (there are two such representations), we are left to examine only the classes $8a$, $8b$ and $8c$. Using MAGMA, we see that the invariant factors of an element $g$ of order $8$ are either $\{(x-1)^6$, $(x-1) ^6\}$, or $\{(x-1)^5$, $(x-1)^7\}$, or $\{(x-1)^4$, $(x-1)^8\}$. Hence $g$ is not almost cyclic.

\subsection{$G=6^.Suz$} 
$G$ can be generated by two conjugates from any of the classes $dX \neq 3c, 3d, 3e$, and by four conjugates from any of the classes $dX=3c,3d,3e$. 

Taking into account  \cite{J} and Lemma \ref{2.1},  we obtain the following:

(i) Whenever $\ell$ does not divide $d$, the Brauer character tables, via the GAP routines, show that cyclic or almost cyclic elements occur exactly according to the following table:

\begin{center} 
\begin{tabular}{c|c|c|c}
$\ell$  & $\dim \Phi$ & $dX$ & type \\ \hline\hline
$\ell \nmid |G|$ & $12$ & $11a$ & almost cyclic\\ 
    & $12$ & $13a,b$ & cyclic \\ \hline

$5$ & $12$ & $11a$ & almost cyclic \\
    & $12$ & $13a,b$ & cyclic \\\hline
$7$ & $12$ & $11a$ & almost cyclic \\
    & $12$ & $13a,b$ & cyclic \\\hline
$11$ & $12$ & $13a,b$ & cyclic \\ \hline
$13$ & $12$ & $11a$ & almost cyclic \\
\end{tabular} 
\end{center}

(ii) We are left to examine only the class $7a$ for $\ell=7$, the class $11a$ for $\ell=11$ and the classes $13a$ and $13b$ for $\ell=13$.

 Let us denote by $a$ and $b$ the 'standard generators' of $6^.Suz$, of order respectively $4$ and $3$, given in \cite{AOL}.

If $\ell=7$, by \cite{HM} $\Phi$ must have degree $12$ (two representations). Pick $g=((ab)^3 a^2bab)^6$. Then $g$ ha order $7$ and its invariant factors are $\{(x-1)^6$, $(x-1)^6\}$. So $g$ is not almost cyclic.

If $\ell=11$, by \cite{HM} $\Phi$ must have degree $12$ (two representations). Thus, an element $g$ of order $11$ is almost cyclic, by Corollary \ref{l+1}.

If $\ell=13$, by \cite{HM} $\Phi$ must have degree $12$ (two representations). Thus, an element $g$ of order $13$ is cyclic, by Proposition \ref{Zal}.\\

\medskip


\subsection{$G=3^.J_3$}



$G$ can be generated by two conjugates from any of the classes $dX$.

By \cite{HM}, either $\dim \Phi=9$, or $\dim \Phi=18$, or $\dim \Phi\geq 126$. In the latter case, in view of Lemma \ref{2.1} we get a contradiction (for all the classes $dX$). So, we may assume that either $\dim \Phi=9$ or $\dim \Phi= 18$.  Taking into account  \cite{J} and Lemma \ref{2.1},  we obtain the following:

(i) Whenever $\ell$ does not divide $d$, the Brauer character tables, via the GAP routines, show that cyclic or almost cyclic elements occur exactly according to the following table:

\begin{center} 
\begin{tabular}{c|c|c|c}
$\ell$  & $\dim \Phi$ & $dX$ & type \\ \hline\hline
$\ell \nmid |G|$ & $18$ & $17a,b$  & almost cyclic\\ 
    & $18$ & $19a,b$ & cyclic \\ \hline
$2$ & $9$ & $9a,b,c$; $17a,b$;  $19a,b$ & cyclic \\
    & $18$ & $17a,b$ & almost cyclic \\
    & $18$ & $19a,b$ & cyclic \\ \hline
$5$ & $18$ & $17a,b$ & almost cyclic \\
    & $18$ & $19a,b$ & cyclic \\ \hline 
$17$ & $18$ & $19a,b$ & cyclic \\ \hline
$19$ & $18$ & $17a,b$ & almost cyclic \\
\end{tabular} 
\end{center}
 
(ii) We are left to examine only the elements of order $8$ for $\ell=2$, of order $5$ for $\ell=5$, of order $17$  for $\ell=17$ and of order $19$ for $\ell=19$.

 Let us denote by $a$ and $b$ the 'standard generators' of $3^.J_3$, of order respectively $2$ and $3$, given in \cite{AOL}. Then:

(1) If $\ell=2$, let $g=((a b^2) ^4 a b)^3$. Then $g$ has order $8$. By Lemma  \ref{2.1}, $\dim \Phi=9$ (there are two such representations). The invariant factors of $g$ are $\{(x-1)$, $(x-1)^8\}$, so $g_8$ is almost cyclic. 

(2) If $\ell=5$, the elements of order $5$ are not almost cyclic by Lemma \ref{2.1}.

(3) If $\ell=17$, by \cite{J} $\Phi$ must have degree $18$ (there are four of these $\Phi$'s), and if $g$ is an element of order $17$, then $g$ is almost cyclic by Corollary \ref{l+1}.

(4) If $\ell=19$, by \cite{J} $\Phi$ must have degree $18$ (there are four of these $\Phi$'s), and the elements of order $19$ are cyclic by Corollary \ref{l+1}.

\subsection{$G=2^. Ru$}

$G$ can be generated by two conjugates from any of  the classes $dX$.




By \cite{HM}, either $\dim \Phi=28$ or $\dim \Phi>250$. In the latter case,  by Lemma \ref{2.1} no $g\in G$ belonging to any of the classes $dX$ can be almost cyclic. So, we may assume that $\dim \Phi=28$ (there are two such representations), and we are only left to examine the elements of order $29$. A computation using GAP if $\ell \neq29$, and Corollary \ref{l+1} if $\ell=29$, show that these elements are cyclic.





\subsection{$G=3^.O'N$}

$G$ can be generated by two conjugates from any of  the classes $dX$.




If $\ell \neq 7$, by \cite{HM} $\dim \Phi\geq 153$, and by Lemma \ref{2.1} no $g\in G$ belonging to any of the classes $dX$ can be almost cyclic. Again by \cite{HM}, if $\ell=7$ then either $\dim \Phi= 45$ (there are two representations of this degree), or $\dim \Phi>250$. As above, the latter case is ruled out by Lemma \ref{2.1}. In the former case,  $g\in G$  cannot be almost cyclic unless it belongs to one of the classes $31a$ and $31b$. However, for all these elements, since $\ell$ does not divide their order, we can inspect the Brauer character tables. They show that none of them can be almost cyclic.





\subsection{$G=3^.Fi_{22}$}

$G$ can be generated by two conjugates from any of the classes $dX\neq3c,3d,3e,3f$, and can be generated by three conjugates from any of the classes $dX=3c,3d,3e,3f$.


Now, by \cite{HM}, if $\ell\neq 2$ $\dim \Phi\geq 351$. Thus, this case is ruled out by Lemma \ref{2.1}. If $\ell=2$, then either $\dim \Phi=27$ or $\dim \Phi>250$. Again, the latter case is ruled out by Lemma \ref{2.1}. In the former case, $g\in G$  cannot be almost cyclic unless it belongs to one of the classes $16a$ and $16b$.

So, assume that $\ell=2$ and $\dim \Phi=27$ (there are exactly two such $\Phi$'s (see  \cite{N})). Let us denote by $a$ and $b$ the 'standard generators' of $3^.Fi_{22}$,  of order respectively $2$ and $13$, given in \cite{AOL}. 
 Note that
  $(16a)^5\in 16b$, 
and pick $g=((ba)^2 b^2a)^3$. Then $g$ has order $16$, and its invariant factors are $\{(x-1)$, $(x-1)^{10}$, $(x^{16} -1)\}$. So $g$ is not almost cyclic.





\subsection{$G=2^. Co_1$}



Here $G$ is generated by two conjugates from any of the classes $dX$, except for the class $3a$, in which case three conjugate generators suffice.

By \cite{HM}, either $\dim \Phi=24$ or $\dim \Phi> 250$. The latter case is ruled out  by Lemma \ref{2.1}. So, assume that $\dim \Phi=24$ (such a representation only occurs if $\ell \neq2$). 

Taking into account  \cite{J} and Lemma \ref{2.1},  we obtain the following:

(i) Whenever $\ell$ does not divide $d$ and $\ell \neq 3,5,$ the Brauer character tables, via the GAP routines, show that cyclic or almost cyclic elements occur exactly according to the following table:

\begin{center} 
\begin{tabular}{c|c|c|c}
$\ell$  & $\dim \Phi$ & $dX$ & type \\ \hline\hline
$\ell \nmid |G|$ & $24$ & $23a,b$ & almost cyclic \\ \hline 
$7$ & $24$ & $23a,b$ & almost cyclic \\ \hline 
$11$ & $24$ & $23a,b$ & almost cyclic \\ \hline 
$13$ & $24$ & $23a,b$ & almost cyclic \\ 
\end{tabular} 
\end{center}

(ii) If $\ell=3$ or $5$, GAP does not give information on the relevant Brauer characters. However, we may still obtain the desired answers.

First of all observe  that, in view of Lemma \ref{2.1}, we only need to examine elements of $G$ of order $13$, $16$ and $23$. Furthermore, $G$ has three classes of elements of order $16$: a class $16a$ whose elements have centralizer of order $2^6$ and two classes $16b$ and $16c$ whose elements have centralizers of order $2^7$. Also, note that $(23a)^5\in 23b$.

Now, let us denote by $a$ and $b$ the 'standard generators' of $2^.Co_1$,  of order respectively $4$ and $3$, given in \cite{AOL}.  The following holds:

(a) Set  $g_{13}=(b^2 ab^2 a^3b  ab^2a b^2 a b a b^2 ab a^3 ba)^3$  and $g_{23}=(ba)^2 b^2 a^3 b^2 ab^2 a^2 (a b^2)^3 a b a b^2 ab a^3ba$. Then each $g_i$ has exactly order $i$, and we see that
the invariant factors of $g_{13}$ are $\{(x-1)^{12}$, $(x-1)^{12}\}$ and those of  $g_{23}$ are $\{(x-1)$, $(x-1)^{23}\}$. So $g_{23}$ is almost cyclic, whereas $g_{13}$ is not.

(b)  The element $g=(ab)^5(ab^2)(ab)$ has order $16$ and its centralizer has order $2^6$, so it belongs to the class $16a$. This element is not almost cyclic, since it has minimum polynomial $m_g(x)=\frac{x^{16}-1}{x^2+1}$ and characteristic polynomial $p_g(x)=\frac{(m_g(x))^2}{x^4+1}$. Next, let $g=(bab^2 a)^3 (ba)^4  ba^2 (ba)^2 b^2 ab  (a b^2)^3 a b a(ab)^5  ab^2ab$. It can be checked that both $g$ and $a^2g$ have order $16$ and centralizer of order $2^7$. Moreover, they are not conjugate to each other; hence, they are representatives of the classes $16b$ and $16c$. Neither of them is almost cyclic, since for both of them the minimum polynomial is $m(x)=\frac{x^{16}-1}{x+1}$, whereas the characteristic polynomial is $m(x)(x-1)(x^8+1)$.


(iii) We are now left to examine the cases $\ell=d$ for $\ell=13$  and $23$. We get the following:


If $\ell=13$,  the invariant factors of $g_{13}$ are $\{(x-1)^{12}$, $(x-1)^{12}\}$. So this element is not almost cyclic.

If $\ell=23$: the elements of order $23$ are almost cyclic by Corollary \ref{l+1}.

\section{Cyclic and almost cyclic elements in the representations of sporadic groups: the results}\label{repres}

The results that we have obtained are assembled in the following Theorems, whose proof is embodied in the analysis carried out in Sections \ref{spor} and \ref{covspor}:
\begin{theor}\label{cyclic}
Let $G$ be a quasi-simple finite sporadic group, $F$ be an algebraically closed field of characteristic $\ell$ and $\Phi$ be an irreducible faithful representation of $G$ over $F$. Let $g$ be an element of $G$ and suppose that $g=g_1z$, where  $g_1 \in G$ has prime-power order $d>1$ and $z \in Z(G)$. Denote by $dX$ the conjugacy class of $g_1$. Then $\Phi(g)$ is cyclic if and only if one of the cases listed in the following table occurs:

\begin{center}
\begin{tabular}{c|c|c|c}
$G$ & $\dim \Phi$ & $dX$ & $\ell$ \\ \hline \hline
$M_{11}$ & $5$ & $5a; 8a,b; 11a,b$ & $3$\\
{} & $9$ & $11a,b$ & $11$\\ 
{} & $10$ & $11a,b$ & any\\
{} & $11$ & $11a,b$ & $\neq2,3$\\ \hline
$M_{12}$ & $10$ & $11a,b$ & $2,3$ \\ 
{} & $11$ & $11a,b$ & $\neq2,3$ \\ \hline
$2^.M_{12}$ & $6$ & $8a,b,c,d; 11a,b$ & $3$ \\ 
{} & $10$ & $11a,b$ & $\neq2$ \\ \hline
$M_{22}$ & $10$ & $11a,b$ & $2$\\ \hline
$2^.M_{22}$ & $10$ & $11a,b$ & $\neq2$\\ \hline
$3^.M_{22}$ & $6$ & $7a,b$; $8a$; $11a,b$ & $2$\\ \hline
$M_{23}$ & $11$ & $11a,b; 23a,b$ & $2$ \\ 
{} & $21$ & $23a,b$ & $23$ \\ 
{} & $22$ & $23a,b$ & $\neq2,23$ \\ \hline
$M_{24}$ & $11$ & $11a; 23a,b$ & $2$ \\ 
{} & $22$ & $23a,b$ & $3$ \\ 
{} & $23$ & $23a,b$ & $\neq2,3$ \\ \hline
$J_1$ & $7$ & $7a$; $11a$;  $19a,b,c$ & $11$ \\ \hline
$J_2$ & $6$ & $7a;8a$ & $2$\\ \hline
$2^.J_2$ & $6$ & $7a$; $8a,b$ & $\neq2$ \\ \hline
$J_3$ & $18$ & $19a,b$  & $3$ \\ \hline
$3^.J_3$ & $9$ & $9a,b,c;17a,b;19a,b$ & $2$ \\ 
{} & $18$ & $19a,b$ & $\neq3$ \\ \hline
$Ru$ & $28$ & $29a,b$ & $2$\\ \hline
$2^.Ru$ & $28$ & $29a,b$ & $\neq2$\\ \hline
$2^.Suz$ & $12$ & $13a,b$ & $3$ \\ \hline
$3^.Suz$ & $12$ & $13a,b$ & $2$ \\ \hline
$6^.Suz$ & $12$ & $13a,b$ & $\neq2,3$ \\ \hline
$Co_3$ & $22$ & $23a,b$ & $2,3$ \\ 
{} & $23$ & $23a,b$ & $\neq2,3$ \\ \hline
$Co_2$ & $22$ & $23a,b$ & $2$ \\ 
{} & $23$ & $23a,b$ & $\neq2$ \\
\end{tabular}
\end{center}

\end{theor}

\begin{theor}\label{almost}
Let $G$ be a quasi-simple finite sporadic group, $F$ be an algebraically closed field of characteristic $\ell$ and $\Phi$ be an irreducible faithful representation of $G$ over $F$. Let $g$ be an element of $G$ and suppose that $g=g_1z$, where  $g_1 \in G$ has prime-power order $d>1$ and $z \in Z(G)$. Denote by $dX$ the conjugacy class of $g_1$. Then $\Phi(g)$ is almost cyclic, but not cyclic, if and only if one of the cases listed in the following table occurs:

\begin{center}
\begin{tabular}{c|c|c|c}
$G$ & $\dim \Phi$ & $dX$ & $\ell$ \\ \hline \hline
$M_{11}$& $5$ & $4a$ & $3$ \\
& $9$ & $8a,b$ & $11$ \\ \hline
$2^.M_{12}$ & $6$ & $5a$ & $3$  \\
& $12$ & $11a,b$ & $\neq 2,3$ \\ \hline
$3^.M_{22}$ & $6$ & $5a$ & $2$\\ \hline
$J_1$ & $14$ & $19a,b,c$ & $11$ \\
& $20$ & $19a,b,c$ & $2$\\ \hline
$J_3$ & $18$ & $17a,b$ & $3$ \\ \hline
$3^.J_3$ & $9$ & $8a$ & $2$ \\
{} & $18$ & $17a,b$ & $\neq 3 $\\\hline
$3^.Suz$ & $12$ & $11a$ & $2$ \\ \hline
$6^.Suz$ & $12$ & $11a$ & $\neq 2,3$ \\ \hline
$Co_1$ & $24$ & $23a,b$ & $2$\\ \hline
$2^.Co_1$ & $24$ & $23a,b$ & $\neq 2 $\\
\end{tabular}
\end{center}

\end{theor}

\bigskip

\section*{Appendix}

\medskip

It is well known that  the group $G=Co_1$ has an irreducible representation $\Phi$ of degree $24$ in characteristic $2$. However, it seems that the uniqueness of such a representation has not yet  been settled in the existing literature. Here we prove the uniqueness of such a $\Phi$ working out its restrictions to certain maximal subgroups $H_i$'s of $G$, whose Brauer character table is known for $\ell=2$.  In this way, we show that $G$ has a unique irreducible Brauer character $\phi$ of degree $24$ for $\ell=2$. We keep the notation of GAP for subgroups, irreducible Brauer characters and classes.\\

Let us first consider the maximal subgroup $H_1=Co_2$ and its irreducible Brauer characters for $\ell=2$. This group has (obviously) a unique character $\chi_1$ of degree $1$ and a unique character $\chi_2$ of degree $22$, while the other characters have degree greater than $24$. So, necessarily, $\phi|_{H_1}=2\cdot \chi_1+\chi_2$.

Next, consider $H_2=3.Suz.2$. This group has a unique character $\chi_1$ of degree $1$ and a unique character $\chi_{14}$ of degree $24$, while the other characters have degree greater than $24$. So, necessarily, $\phi|_{H_2}=\chi_{14}$.

Using the previous two subgroups we can determine the value of $\phi$ on the classes $7a$ and $7b$ of $G$  (since $H_1$ contains elements of the class $7b$ and $H_2$ contains elements of the class $7a$). We get that $\phi(7a)=-4$ and $\phi(7b)=3$. These values will be used in the following.

Now, consider the subgroup $H_6=U_6(2).3.2$. This group has a unique character $\chi_1$ of degree $1$, a unique character $\chi_{2}$ of degree $2$ and a unique character $\chi_3$ of degree $20$, while the other characters have degree greater than $24$. Furthermore, it contains elements of the class $7b$ of $G$ (labelled $7a$ in $H_6$) for which $\chi_1(7a)=1$, $\chi_2(7a)=2$ and $\chi_3(7a)=-1$. So, $\chi_3$ must be a component of $\phi|_{H_6}$. We have three possible decompositions fo $\phi|_{H_6}$ : $4\cdot \chi_1+\chi_3$, $2\cdot \chi_1+\chi_2+\chi_{3}$ or  $2\cdot \chi_2+\chi_3$.
Looking at the classes $3a$, $3c$ and $3e$ of $H_6$, we have a priori the following possibilities:

\begin{center}
 \begin{tabular}{l|ccc}
 $\phi|_{H_6}$ & $3a$ & $3c$ & $3e$\\\hline
  $4\cdot \chi_1+\chi_3$ & $6$ & $6$ & $12$ \\
  $2\cdot \chi_1+\chi_2+\chi_{3}$ & $6$ & $6$ & $9$ \\
   $2\cdot \chi_2+\chi_3$ & $6$ & $6$ & $6$ \\
 \end{tabular}
 \end{center}

However,  the classes $3a$, $3c$ and $3e$ of $H_6$ are fused into the class $3b$ of $G$. This forces $\phi|_{H_6}=2\cdot \chi_2+\chi_3$.

Next, consider the subgroup $H_7=(A_4\times G_2(4)):2$.  This group has a unique character $\chi_1$ of degree $1$, a unique character $\chi_{11}$ of degree $2$, and three characters $\chi_2$, $\chi_{12}$, $\chi_{13}$ of degree $12$, while all the other characters have degree greater than $24$. Moreover, this subgroup contains elements belonging to the class $7a$ of $G$ (also labelled $7a$ in $H_7$) for which $\chi_1(7a)=1$, $\chi_{11}(7a)=2$ and $\chi_2(7a)=\chi_{12}(7a)=\chi_{13}(7a) = -2$. 
So, necessarily,  $\phi|_{H_7}$ can only have components of degree $12$, and hence $6$ possible distinct decompositions.

Looking at the classes $3a$, $3c$, $15e$ and $15f$ of $H_7$, we have a priori the following possibilities:

\begin{center}
 \begin{tabular}{l|cc|cc}
 $\phi|_{H_7}$ & $3a$ & $3c$ & $15e$ & $15f$ \\\hline
  $ \chi_2+\chi_{12}$ & $-12$ & $6$  & $(-3-i\sqrt{15})/2$ & $(-3+i\sqrt{15}/2$ \\
  $\chi_2+\chi_{13}$ & $-12$ & $6$ & $(-3+i\sqrt{15})/2$ &  $(-3-i\sqrt{15})/2$ \\
   $\chi_{12}+\chi_{13}$ & $-12$ & $-12$ & $3$ & $3$ \\
   $2\cdot \chi_2$ & $-12$ & $24$ & $-6$ & $-6$ \\
   $2\cdot \chi_{12}$ & $-12$ & $-12$ &  $3-i\sqrt{15}$ & $3+i\sqrt{15}$ \\
   $2\cdot \chi_{13}$ & $-12$ & $-12$ & $3+i\sqrt{15}$ & $3-i\sqrt{15}$\\
 \end{tabular}
 \end{center}

However, the classes $3a$ and $3c$ of $H_7$ are fused into the class $3a$ of $G$ and the classes $15e$ and $15f$ of $H_7$  are fused into the class $15a$ of $G$. 

 This forces $\phi|_{H_7}=\chi_{12}+\chi_{13}$.\\

 Finally, let us consider the subgroup $H_{12}=(A_5\times J_2):2$. This group has a unique character $\chi_1$ of degree $1$, two characters, $\chi_8$ and $\chi_{18}$, of degree $4$, a unique character $\chi_2$ of degree $12$, two characters, $\chi_9$ and $\chi_{10}$, of degree $24$, while all the other characters have degree greater than $24$.  Furthermore, this subgroup contains elements belonging to the class $7a$ of $G$ (also labelled $7a$ in $H_{12}$) for which $\chi_1(7a)=1$, $\chi_{8}(7a)=\chi_{18}(7a)=4$, $\chi_2(7a)=-2$ and $\chi_9(7a)=\chi_{10}(7a)=-4$. So, we have three possible decompositions for $\phi|_{H_{12}}$: $2\cdot \chi_2$, $\chi_9$ or $\chi_{10}$.
Looking at the classes $3a$, $3c$, $5b$, $5c$ and $5d$ of $H_{12}$, we have a priori the following possibilities:

 \begin{center}
 \begin{tabular}{l|cc|ccc}
 $\phi|_{H_{12}}$ & $3a$ & $3c$ & $5b$ & $5c$ & $5d$ \\\hline
  $ 2\cdot \chi_2$ & $-12$ & $24$ & $-6$ & $24$ & $4$\\
  $\chi_9$ & $-12$ & $-12$ & $-6$ & $-6$ & $-6$\\
   $\chi_{10}$ & $-12$ & $-12$ & $-6$ & $-6$ & $4$\\
 \end{tabular}
 \end{center}
 
   However, the classes $3a$ and $3c$ of $H_{12}$ are fused into the class $3a$ of $G$ and the classes $5b$, $5c$ and $5d$ of $H_{12}$  are fused into the class $5a$ of $G$. 

  This forces $\phi|_{H_{12}}=\chi_9$.
  \medskip 
 
 Now, since every conjugacy class of $G$ has a non-trivial intersection with at least one of the maximal subgroups considered above, we conclude that the value of $\phi$ is uniquely determined. In other words, $G$ has a unique irreducible representation of degree $24$ in characteristic $2$.

\medskip

Authors' addresses:

\medskip

\author{
L. DI MARTINO: 
\address{Universit\`a degli Studi di Milano-Bicocca, Dipartimento
di Matematica e Applicazioni, via R. Cozzi 53,  20125 Milano,
Italy. e-mail:  lino.dimartino@unimib.it}

\author{M.A. PELLEGRINI}:
\address{Departamento de Matem\'atica, Universidade de Bras\'ilia, 70910-900 Bras\'ilia - DF, Brazil}. e-mail: pellegrini@unb.br}

\author{A.E. ZALESSKI\address{:
\address{Universit\`a degli Studi di Milano-Bicocca, Dipartimento
di Matematica e Applicazioni, via R. Cozzi 53,  20125 Milano,
Italy. e-mail: alexandre.zalesskii@gmail.com}

\end{document}